\documentclass[a4,12pt]{amsart}

\usepackage{float}

\usepackage[usenames]{color}
\usepackage{amsfonts}
\usepackage{amssymb}
\usepackage[normalem]{ulem}

\usepackage{soul}
\usepackage{amsmath,amssymb,amsthm,amsfonts,enumerate,url,mathrsfs,tikz,graphicx}
\usepackage{hyperref}
\usepackage{graphicx}

\usetikzlibrary{automata,positioning}
\usetikzlibrary{calc}

\usepackage[utf8]{inputenc}

\usepackage{verbatim}
\usepackage{array}

\usepackage{bbm,dsfont}
\usepackage[all]{xy}

\newtheorem{lemma}{Lemma}[section]
\newtheorem{conj}[lemma]{Conjecture}
\newtheorem{exa}[lemma]{Example}
\newtheorem{cor}[lemma]{Corollary}
\newtheorem{defi}[lemma]{Definition}
\newtheorem{theo}[lemma]{Theorem}
\newtheorem{rem}[lemma]{Remark}
\newtheorem{prop}[lemma]{Proposition}

\numberwithin{equation}{section}
\numberwithin{lemma}{section}

 \def\mG{\mathsf{G}}

 \def\mV{\mathsf{V}}
 \def\mE{\mathsf{E}}

 \def\mP{\mathsf{P}}

 \def\mv{\mathsf{v}}
 \def\me{\mathsf{e}}
 \def\mf{\mathsf{f}}
 \def\mw{\mathsf{w}}

 \def\mf{\mathsf{f}}

\newcommand{\R}{\mathbb{R}}

\newcommand{\N}{\mathbb{N}}

\newcommand{\D}{\mathcal{D}}

\newcommand\Graph{\mathcal{G}}

\DeclareMathOperator{\Vol}{Vol}
\DeclareMathOperator{\sym}{sym}

\newcommand{\diff}{\ensuremath{\mathop{}\!\mathrm{d}}}

\title{Lower estimates on eigenvalues of quantum graphs} 

\subjclass[2010]{}

\keywords{}

\author[D.~Mugnolo]{Delio Mugnolo}
\author[M.~Plümer]{Marvin Plümer}

\address{Delio Mugnolo, Lehrgebiet Analysis, Fakult\"at Mathematik und Informatik, Fern\-Universit\"at in Hagen, D-58084 Hagen, Germany}
\email{delio.mugnolo@fernuni-hagen.de}

\address{Marvin Plümer, Lehrgebiet Analysis, Fakult\"at Mathematik und Informatik, Fern\-Universit\"at in Hagen, D-58084 Hagen, Germany}
\email{marvin.pluemer@fernuni-hagen.de}

\subjclass[2010]{34B45 (05C50 35P15 81Q35)}

\keywords{Spectral geometry of quantum graphs, Planar graphs, Double Cover Conjecture, Normalized Laplacians}

\thanks{
The authors were partially supported by the Deutsche Forschungsgemeinschaft (Grant 397230547).
}

\begin{document}

\begin{abstract}
A  method for estimating the spectral gap along with higher eigenvalues of non-equilateral quantum graphs has been introduced by Amini and Cohen-Steiner in~\cite{AmiCoh18} recently: it is based on a new transference principle between discrete and continuous models of a graph. We elaborate on it by developing a more general transference principle and by proposing alternative ways of applying it. To illustrate our findings, we present several spectral estimates on planar metric graphs that are oftentimes sharper than those obtained by isoperimetric inequalities and further previously known methods.
\end{abstract}

\maketitle

\section{Introduction}

Our aim in this paper is to present some new developments in the theory of Laplacians on metric graphs: in particular, we discuss the role played by planarity and related concepts in spectral theory.

Intuitively, a metric graph is a collection of $E$ intervals $(0,\ell_1),\ldots,(0,\ell_E)$ whose endpoints are identified in a graph-like fashion, so that each interval is regarded as an edge and the identified endpoints as vertices; a more precise definition can be found e.g.\ in the monographs~\cite{BerKuc13,Mug14}, but this will suffice for our purposes. In this note, we are restricting ourselves to the case of \emph{finite} metric graphs of \emph{finite} total length, i.e., 
\begin{equation}\label{eq:comp}
E<\infty\quad\hbox{and}\quad L:=\sum_{j=1}^E \ell_j<\infty.
\end{equation}

Accordingly, each metric graph $\Graph$ induces a combinatorial graph $\mG$ (with vertex set $\mV$ and edge set $\mE$, $E=|\mE|$) that describes its connectivity; conversely, the underlying combinatorial graph $\mG$ and the vector of edge lengths $(0,\ell_\me)_{\me\in\mE}$ fully determine the metric graph $\Graph$: for this reason, it is natural to consider a weighted version of $\mG$ each of whose edges $\me$ is assigned the weight $\ell_\me$. If no confusion is possible, we tacitly identify $\mG$ with its weighted version.

 At the same time, $\Graph$ can be turned into a metric measure space by 
\begin{itemize}
\item endowing each interval with the Lebesgue measure and
\item extending canonically the usual combinatorial distance defined on $\mG$;
\end{itemize}
again, we refer to~\cite[Chapter 3]{Mug14} for details. In this way, it is natural to introduce the Hilbert space
\[
L^2(\Graph)\equiv \bigoplus_{j=1}^E L^2(0,\ell_j)
\]
of square-integrable functions over the metric graph. An element of $L^2(\Graph)$ is hence a vector $f\equiv(f_1,\ldots,f_E)$ of functions with $f_j\in L^2(0,\ell_j)$ for each $j$; one can then define differential operators edgewise. Especially the operator
\[
A:(f_1,\ldots,f_E)\mapsto (f_1'',\ldots,f_E'')
\]
has enjoyed much attention over the last decades, beginning with the pioneering investigations in~\cite{RueSch53}. 
Throughout this paper we adopt the convention that the Laplacian is positive semidefinite, i.e., $\Delta:=-A$. It is easy to see that 
\[
A_{|\bigoplus_{j=1}^E H^2_0(0,\ell_j)}
\]
is not essentially self-adjoint and $2E$ transmission conditions have to be imposed in the vertices of the graph in order to determine a self-adjoint extension. On any given metric graph, uncountably many self-adjoint extensions of $A$ exist, the most canonical one being defined by imposing continuity across the vertices along with a Kirchhoff-type condition (i.e., for each vertex $\mv$ the sum of all normal derivatives evaluated along edges incident in $\mv$ has to vanish) on elements of $\bigoplus_{j=1}^E H^2(0,\ell_j)$. Equipped with these natural vertex conditions, $\Delta$ is often referred to as the \emph{standard} Laplacian. We will only study this Laplacian realization in the following.

Finite metric graphs of finite total length share an important property with compact manifolds without boundary, domains with Neumann boundary conditions, and finite combinatorial graphs: the standard Laplacian is a self-adjoint, positive semidefinite operator and -- as long as these underlying metric graphs are connected -- the Laplacian's null space is one-dimensional, as it coincides with the space of all constant functions. Furthermore, the standard Laplacian is associated with a Dirichlet form, hence it generates a Markovian semigroup: in the context of metric graphs, this has been proved in~\cite{KraMugSik07}.

Because the embedding of $H^2(\Graph)$ into $L^2(\Graph)$ is compact -- in fact, even of trace class -- the operator $\Delta$ has pure point spectrum consisting of countably many eigenvalues 
\[
0=\lambda_1(\Graph)\le \lambda_2(\Graph)\le \ldots
\]
accumulating at $\infty$; we have already mentioned that $\lambda_1(\Graph)$ is simple and hence $\Delta$ has a spectral gap, i.e., a positive distance between the spectral bound $\lambda_1(\Graph)=0$ and the second-lowest eigenvalue $\lambda_2(\Graph)>0$. In particular, as $t\to\infty$ the semigroup $(e^{-t\Delta})_{t\ge 0}$ generated by $-\Delta$ converges to the orthogonal projector onto the null space exponentially, at a rate given by $e^{-t\lambda_2(\Graph)}$ (\cite[Cor.~5.2]{KraMugSik07}). Many investigations have been devoted in recent years to the issue of estimating $\lambda_2(\Graph)$ and hence the convergence rate of the semigroup: apart from the pioneering work~\cite{Nic87}, we refer in particular to~\cite{KenKurMal16,Ari16,BanLev17,BerKenKur17,BerKenKur19,KosNic19}.
All these articles aim at estimating $\lambda_2(\Graph)$ based on different quantities of combinatorial and/or metric nature, including total length, diameter, total number of vertices or edges, Cheeger constants, etc.

Refined results are furthermore available provided the metric graph belongs to specific classes: we mention e.g.\ the upper estimate in~\cite{Roh16}, which holds for \emph{trees}: these are by definitions metric graphs without cycles, hence have poor connectivity. This paper has in a certain sense an opposite goal: we are going to prove a family of \emph{lower} estimates if suitable assumptions ensuring \emph{rich} connectivity are imposed  on the combinatorial graph $\mG$ underlying $\Graph$.

Our results are essentially based on a transference principle recently developed in~\cite{AmiCoh18} in order to deduce spectral estimates for discrete Laplacians from corresponding estimates known to hold for Laplace--Beltrami operators on manifolds of given genus; it was observed already in~\cite{AmiCoh18} that this transference principle applies to metric graphs as well, and we are going to elaborate on this idea.

The method in~\cite{AmiCoh18} is based on the introduction of a double cover of the metric graph $\Graph$, and in turn of a suitable vicinity graph.
 In Section~\ref{sec:gen-set} we first generalize the basic lower estimate by Amini and Cohen-Steiner by analyzing $m$-fold covers $\mathfrak U$ of $\Graph$, $m\ge 2$: both the elements of $\mathfrak U$ and the associated vicinity graph 
 play a central role and different choices of $\mathfrak U$ lead to different estimates.
 This is indeed the starting point of our analysis: while a smart choice for $\mathfrak U$ was made in~\cite[\S~3]{AmiCoh18} that applies to all metric graphs, we take a closer look at the class of metric graphs whose underlying combinatorial graph $\mG$ admits a \emph{cycle double cover}. For our purposes, such (combinatorial) cycle double covers induce canonically a (metric) double cover $\mathfrak U$. To illustrate our idea and compare it with the original one in~\cite{AmiCoh18}, we thus focus on two prototypical classes of graphs: regular polyhedra and pumpkin chains in Sections~\ref{sec:cdc} and~\ref{sec:pumpk}, respectively. We show by means of numerous examples that spectral estimates based on the Amini--Cohen-Steiner approach (both in the original and in our own version) can often outperform other methods based either on the isoperimetric inequality
 \begin{equation}\label{eq:fried}
 \lambda_k(\Graph) \ge \frac{\pi^2 k^2}{4L^2},\qquad k\ge 2,
 \end{equation}
\cite{Nic87,Fri05} and its refinements in~\cite{Sol02,KurNab14,BanLev17,BerKenKur17} or on surgery methods \cite{KenKurMal16,Ken20,BerKenKur19}.

\section{General Setting and the main result}\label{sec:gen-set}

All spectral inequalities quoted in the introduction aim at estimating the spectral gap $\lambda_2(\Graph)$ of a metric graph $\Graph$ in dependence of one or more metric or combinatorial quantities of $\Graph$: i.e., if only partial knowledge of $\Graph$ is available. The classical result discovered by von Below in~\cite{Bel85} is conceptually different in that it does not deliver estimates on $\lambda_2(\Graph)$, but rather it allows to determine the spectral gap $\lambda_2(\Graph)$ of the metric graph Laplacian with natural boundary conditions if the spectral gap of the normalized Laplacian on the underlying combinatorial graph is given; and more generally, to determine the whole spectrum of the standard Laplacian on $\Graph$, provided perfect knowledge of $\Graph$ is available.

Very recently, a new method based on a similar but different transference principle has been proposed by Amini and Cohen-Steiner in~\cite{AmiCoh18}: unlike von Below's formula, it applies to possibly non-equilateral graphs, but at the price of a comparatively rough estimate.
Their result is based on the introduction of a double cover $\mathfrak U$ of the metric graph $\Graph$ and the associated vicinity graph $\Gamma$. However, their results can be generalized to a larger class of covers of \(\Graph\): to this purpose, we need to introduce a few notions.

\subsection{$m$-fold covers of metric graphs}

Beyond the standing finiteness and compactness assumptions~\eqref{eq:comp}, for technical reasons we also impose that the metric graph $\Graph$ has no loops. This is hardly restrictive, since, after inserting dummy vertices, we can always produce a new, loopless metric graph $\Graph'$ with same topology and same spectrum as $\Graph$. Parallel edges are  allowed, though.

\begin{defi}\label{def:mfold-cover}
Let \(m\in\N\). An \emph{\(m\)-fold cover} of a metric graph $\Graph$ is a finite family $\mathfrak U:=(\mathcal  {\mathcal U}_i)_{1\le i\le k}$ of connected metric subgraphs of $\Graph$ such that for almost every $x\in \Graph$ there exist \(m\) distinct indices $1\leq i_1<\ldots<i_m\leq k$ such that\ $x\in \mathcal U_{i_1}\cap \ldots\cap\mathcal U_{i_m}$ and $x\not\in\mathcal  {\mathcal U}_i$ for \(i\notin\{i_1,\ldots,i_m\}\).

 The \emph{associated vicinity graph} $\Gamma$ is a simple, unoriented weighted graph with vertex set \(\{1,\ldots,k\}\) and edge weights \(\mu_{ij}:=|\mathcal  {\mathcal U}_i\cap \mathcal U_j|\) for vertices \(i\neq j\) and \(\mu_{ii}=0\); in particular two vertices $i\neq j$ are adjacent in \(\Gamma\) if and only if $\mathcal  {\mathcal U}_i\cap \mathcal U_j$ is not a null set with respect to the Lebesgue measure on \(\Graph\).
\end{defi}  

\begin{rem}
	Since \(\mathfrak U\) is an \(m\)-fold cover of \(\mathcal{G}\) the \emph{weighted degree} of a vertex \(i\) of \(\Gamma\) is
	\begin{equation}\label{eq:weighted-degree}
		d_i^\mu:=\sum_{j=1}^k\mu_{ij}=\sum_{\substack{j=1\\j\neq i}}^k|{\mathcal U}_i\cap \mathcal U_j|=(m-1)|{\mathcal U}_i|.
	\end{equation}
	Moreover, the \emph{total volume} of \(\Gamma\) with respect to the vertex degree weight is
	\begin{equation}		
		\Vol_\mu(\Gamma):=\sum_{i=1}^k d_i^\mu=(m-1)\sum_{i=1}^k|{\mathcal U}_i|=m(m-1)L.
	\end{equation}
\end{rem}

\subsection{Normalized Laplacians of weighted graphs}

Given a simple (unoriented) weighted graph $G$ with edge weights \(\mu_\me\), let $M$ be the diagonal matrix whose entries are the edge weights of $G$: we can then introduce the (weighted) normalized Laplacian
\begin{equation}\label{eq:def-normlapl}
\Lambda:=D^{-1}\mathcal I M \mathcal I^T,
\end{equation}
 on $G$, where $\mathcal I$ is the signed incidence matrix of an arbitrary orientation of $G$ (cf.~\cite[\S~2.1]{Mug14}) and $D$ is the  diagonal matrix whose entries are the weighted vertex degrees $d_\mv$ defined for each vertex $\mv$ by
 \[
 d_\mv:=\sum_{\me\sim\mv}\mu_\me.
 \]
A different, popular linear operator on the vertex space of a graph is
\begin{equation}\label{eq:def-symnormlapl}
	\Lambda_{\sym}:=D^{-\frac{1}{2}}\mathcal I M \mathcal I^TD^{-\frac{1}{2}},
\end{equation}
which is often referred to as the \emph{symmetric normalized Laplacian}. Because \(\Lambda\) and \(\Lambda_{\sym}\) are unitarily equivalent, they have the same -- real -- eigenvalues: we denote them by $\alpha_i(G)$, with
\[
0=\alpha_1(G)\leq\alpha_2(G)\le \ldots \le \alpha_V(G)\le 2.
\]

In the following we are also going to consider graphs with parallel edges; for this purpose, we need to generalize the above notion. If $G$ is a weighted graph  with parallel edges, let us consider its \textit{reduced graph}, i.e., the simple weighted graph $\tilde{G}$ defined as follows:
\begin{itemize}
\item $\tilde{G}$ has  same vertex set as $G$;
\item given any two adjacent vertices $\mv,\mw$ in $G$, $\mv,\mw$ are adjacent in $\tilde{G}$, too, and the weight of the corresponding edge is the sum of weights of all edges connecting $\mv,\mw$ in $G$.
\end{itemize}
Then, the normalized Laplacian of a general weighted, loopless graph $G$ is defined as the normalized Laplacian of the reduced graph $\tilde{G}$; again, we denote its eigenvalues by $\alpha_i(G)$, in increasing order. Note that \(G\) is connected, if and only if \(\alpha_2(G)>0\).

\subsection{The Amini--Cohen-Steiner Theorem}

Let us present the main result in this section.

\begin{theo}\label{thm:amicoh-basic}
Let $\Graph$ be a metric graph. Given an \(m\)-fold cover $\mathfrak U:=(\mathcal  {\mathcal U}_i)_{1\le i\le k}$ of $\Graph$ with associated vicinity graph $\Gamma$, the $k$ lowest eigenvalues of the Laplacian with natural conditions on $\Graph$ satisfy
\begin{equation}\label{eq:amicoh}
\lambda_i(\Graph)\ge \frac{m-1}{m}\eta\, \alpha_i(\Gamma),\qquad i=1,\ldots,k,
\end{equation}
where $\eta$ is the minimal spectral gap of the Laplacian with natural conditions defined on any $\mathcal U_j$, $j=1,\ldots,k$, i.e.,
\[
\eta:=\min_{1\le i\le k}\lambda_2(\mathcal {\mathcal U}_i).
\]
\end{theo}

(Observe that $\eta>0$, since all $\mathcal {\mathcal U}_i$ are supposed to be connected.)

Theorem~\ref{thm:amicoh-basic} generalizes~\cite[Thm.~1.2]{AmiCoh18} from double covers (\(m=2\)) to \(m\)-fold covers. 

\begin{proof}
	Since the proof uses similar arguments as the one given in~\cite{AmiCoh18} for double covers, we restrict ourselves to the main arguments. We consider the linear bounded operator \(\Phi:L^2(\Graph)\rightarrow \R^k\) given by
		\[\left(\Phi f\right)_i :=\frac{1}{\sqrt{|{\mathcal U}_i|}}\int_{{\mathcal U}_i}f\diff x,\qquad i=1,\ldots k\]
	and let \(\Phi^*\) denote its adjoint operator. Using the Courant--Fischer Theorem for the respective eigenvalues and using the fact that \(\mathfrak U\) is an \(m\)-fold cover of \(\Graph\) one can show that
		\[\lambda_i(\mathcal{G})\geq \frac{\eta}{m}\alpha_i(mI-\Phi\Phi^*), \qquad i=1,\ldots k\]
	where \(I\) is the identity on \(\R^k\) and \(\alpha_i(mI-\Phi\Phi^*)\) denotes the \(i\)-th-lowest eigenvalue of the operator \(mI-\Phi\Phi^*\). The entries of \(\Phi\Phi^*\) with respect to Cartesian coordinates on \(\R^k\) are
		\[\left(\Phi\Phi^*\right)_{ij}=\frac{|{\mathcal U}_i\cap \mathcal U_j|}{\sqrt{|{\mathcal U}_i|\,|\mathcal U_j|}}\]
	and, thus, \(\left(mI-\Phi\Phi^*\right)_{ii}=m-1\) and
		\[\left(mI-\Phi\Phi^*\right)_{ij}=-\frac{|{\mathcal U}_i\cap \mathcal U_j|}{\sqrt{|{\mathcal U}_i|\,|\mathcal U_j|}}=-(m-1)\frac{\mu_{ij}}{\sqrt{d_i^\mu d_j^\mu}}\]
	for \(i\neq j\), where we used \eqref{eq:weighted-degree} in the second step. Therefore, the entries of \(mI-\Phi\Phi^*\) are in fact equal to \(m-1\) times the entries of the symmetric normalized Laplacian \eqref{eq:def-symnormlapl} on \(\Gamma\) and we obtain
		\[\lambda_i(\mathcal{G})\geq \frac{\eta}{m}\alpha_i(mI-\Phi\Phi^*) = \frac{m-1}{m}\eta\, \alpha_i(\Gamma),\]
	 for \(i=1,\ldots k\). This proves the claim.
\end{proof}
\begin{rem}\label{rem:tautol}
The factor $\frac{m-1}{m}$ cannot be improved: if the double cover of $\Graph$ consists of \(m\) identical copies $\mathcal U_1,\cdots,\mathcal U_m$ of $\Graph$, then \(\Gamma\) is the -- unweighted -- complete graph on \(m\) vertices and, thus, $\alpha_2(\Gamma)=\frac{m}{m-1}$ (cf.~\cite[Exa.~1.1]{Chu97}), so the inequality in~\eqref{eq:amicoh} is in fact an equality for \(i=2\).
\end{rem}

Of course, \eqref{eq:amicoh} can be improved by taking the supremum over all possible \(m\)-fold covers of $\Graph$. Intuitively, taking smaller $\mathcal  {\mathcal U}_i$'s leads to higher $\eta$ but lower $\alpha_i$, since the vicinity graph tends to get sparser and hence to have poorer connectivity.
It seems that this trade-off is not easy to optimize, even just for $i=2$: $\eta$ can be estimated owing to the known inequalities
\begin{equation}\label{eq:nicaise-kkmm}
\frac{\pi^2}{|\mathcal U|^2}\le \lambda_2(\mathcal U)\le \frac{\pi^2 E_{\mathcal U}^2}{|\mathcal U|^2},
\end{equation}
cf.~\cite{Nic87,KenKurMal16}; here $E_{\mathcal U},E_\Gamma$ are the number of edges of $\mathcal U,\Gamma$, respectively. But little is known about $\alpha_2(\Gamma)$ apart from 
\begin{equation}\label{eq:bounds-normlapl}
\max\left\{\frac{h^2_\Gamma}{2}, \frac{4}{D_{\mu^{-1}}(\Gamma) \Vol_{\mu}(\Gamma)}\right\}\le \alpha_2(\Gamma)\le 2h_\Gamma,
\end{equation}
where $h_\Gamma$, $\Vol_{\mu}(\Gamma)$, and $D_{\mu^{-1}}(\Gamma)$ are the Cheeger constant, the total weight of the weighted graph $\Gamma$ with respect to the vertex degree weight based on the edge weights, and the diameter of \(\Gamma\) with respect to the path metric on \(\Gamma\) induced by the inverse edge weights \(\mu_{ij}^{-1}\), respectively; see~\cite[\S~1.3 and \S~2.3]{Chu97} for unweighted versions of these results and~\cite[Thm.~3.6]{BauKelWoj15} and~\cite[Cor.~3.7]{LenSchSto18} for general versions. (But see~\cite{BerKenKur17,BerKenKur19} for sharper estimates whenever $\mathcal U$ or $\Gamma$ have higher connectivity.)

In~\cite{AmiCoh18} Amini and Cohen-Steiner choose to work with \emph{star double covers}: a star double cover is  a double cover consisting of $V$ stars \(\mathcal S_{\mv_j}\), each centered at a different vertex $\mv_j$ and consisting of all edges incident in $\mv_j$, $j=1,\ldots,V$ (where $V$ is the number of vertices of $\Graph$). Not only is this choice particularly natural because the resulting weighted vicinity graph $\Gamma$ has the same topology as the underlying combinatorial graph \(\mG\) of \(\Graph\) and the edge weight \(\mu_{ij}\) is equal to the length of the edge connecting the vertices \(\mv_i\) and \(\mv_j\) in \(\Graph\); we will show in several examples that it is also surprisingly efficient. Applying Theorem~\ref{thm:amicoh-basic} to star double covers it was proved in~\cite[Thm.~3.4]{AmiCoh18} that
\begin{equation}\label{eq:amicoh-0}
\lambda_i(\Graph)\ge \frac{\pi^2}{8\ell_{\max}^2}\alpha_i(\mG),\qquad i=1,\ldots,V,
\end{equation}
where $\ell_{\max}$ is the maximal length of any edge in $\Graph$. This estimate can be improved by using different lower estimates for the spectral gaps of the single stars: Nicaise' inequality \eqref{eq:nicaise-kkmm} and \cite[Thm. 1.1]{Ken20} imply that 
	\begin{align*}
		\lambda_2(\mathcal S_{\mv_j}) & \geq \frac{\pi^2}{\mathrm{deg}_{\ell,j}^2}, 
		& \lambda_2(\mathcal S_{\mv_j}) & \geq \frac{1}{D_j\cdot\mathrm{deg}_{\ell,j}}
	\end{align*}
for \(j=1,\ldots,V\) where
\(\mathrm{deg}_{\ell,j}  :=|\mathcal S_{\mv_j}|\) is the weighted degree of the vertex \(\mv_j\) and  \(D_j :=\mathrm{diam}(\mathcal S_{\mv_j})\) is the diameter of the star \(\mathcal S_{\mv_j}\), i.e., the total length of all edges incident in $\mv_j$ and the combined length of the two longest edges incident in $\mv_j$, respectively.
Applying Theorem \ref{thm:amicoh-basic} yields:

\begin{prop}\label{prop:aminicoh}
Let $\Graph$ be a metric graph with underlying weighted combinatorial graph $\mG$. Then the eigenvalues $\lambda_i(\Graph)$ of the Laplacian with natural conditions on $\Graph$ satisfy
\begin{equation}\label{eq:amicoh-2}
\lambda_i(\Graph)\ge \max\left(\frac{\pi^2}{8\ell_{\max}^2},\frac{\pi^2}{2\mathrm{deg}_{\ell,\mathrm{max}}^2},\frac{1}{2(D\cdot\mathrm{deg}_\ell)_{\max}}\right)\alpha_i(\mG),\qquad i=1,\ldots,V,
\end{equation}
where \(\ell_{\max}\) is the maximal length of the edges of \(\mathcal G\), $\mathrm{deg}_{\ell,\mathrm{max}}$ is the maximal weighted degree of the vertices of \(\mG\) with respect to the edge lengths of $\Graph$ and
	\[(D\cdot\deg)_{\max}:=\max\left\{D_j\cdot\deg_{\ell,j}~|~j=1,\ldots,V\right\}.\] 
\end{prop}

\begin{rem}
1) Actually~\cite[Thm.~1.2]{AmiCoh18} is only formulated for simple metric graphs, but the proof of Theorem~\ref{thm:amicoh-basic} shows that the assertion remains true if $\mathcal G$ and its subgraphs $\mathcal U_i$ have parallel edges; in this case, it is natural to generalize the notion of \emph{stars} to that of \emph{pumpkin stars}, and accordingly consider pumpkin star double covers rather than star covers.

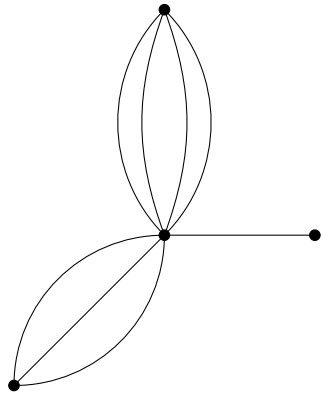
\begin{figure}[H]
\begin{tikzpicture}[scale=1]
\coordinate (a) at (0,0);
\coordinate (b) at (-2,-2);
\coordinate (c) at (2,0);
\coordinate (d) at (0,3);
\draw[fill] (a) circle (2pt);
\draw[fill] (b) circle (2pt);
\draw[fill] (c) circle (2pt);
\draw[fill] (d) circle (2pt);
\draw (a) -- (b);
\draw[bend left=45] (a) edge (b);
\draw[bend right=45] (a) edge (b);
\draw (a) edge (c);
\draw[bend left=45] (a) edge (d);
\draw[bend right=45] (a) edge (d);
\draw[bend left=20] (a) edge (d);
\draw[bend right=20] (a) edge (d);
\end{tikzpicture}
\vspace{5pt}
\caption{ A pumpkin star}
\label{fig:shifting-pumpkins-a}
\end{figure}

2) Formally, \eqref{eq:amicoh} and \eqref{eq:amicoh-2} hold with $\alpha_i(\mG)$ replaced by $\alpha_i(\Gamma)$; but for (pumpkin) star double covers, $\mG$ (or its associated reduced graph, if $\mG$ has parallel edges) and the vicinity graph $\Gamma$ coincide.
\end{rem}

We will show in several examples that the estimate~\eqref{eq:amicoh-2} is often very good in comparison with other known ones. 
Perhaps more interesting even is the flexibility offered by the idea behind Theorem~\ref{thm:amicoh-basic}: instead of the double cover based on stars -- which always exists -- we are going to focus on special classes of graphs that admit a specific, different double cover: the elements of $\mathfrak U$ will not be stars, but rather cycles.

\section{Cycle double covers and regular polyhedra}\label{sec:cdc}


Double covers whose elements are cycles within the graphs are a classical topic of (combinatorial) graph theory.
Not all (metric) graphs admit such double covers: obviously, a necessary condition is that the graph contains no bridges (i.e., it is not simply connected).
Many authors conjectured independently that this condition is sufficient, too. 

\begin{conj}[Cycle double cover conjecture]
Every bridgeless graph 
 has a cycle double cover.
\end{conj}

The cycle double cover conjecture is among the major open problems of graph theory and we will not comment on it: we only refer to the monograph~\cite{Zha12} and the survey~\cite{Jae85} and mention that the conjecture is especially known to hold in the following cases, under the additional assumption that $\mG$ is bridgeless:
\begin{itemize}
\item $\mG$ is planar (i.e., it can be drawn in $\R^2$ without any edge crossings);
\item $\mG$ is complete (i.e., any two vertices are adjacent if and only if they are not equal);
\item $\mG$ contains a Hamiltonian path (i.e., there is a path contained in $\mG$ such that traverses every vertex exactly once);
\item $\mG$ is cubic and 3-edge colorable (i.e., each vertex has exactly three incident edges and there is a way of assigning to  each edge one out of three given colors in such a way that any  two adjacent edges have different colors);
\item $\mG$ is 4-edge colorable.
\end{itemize}

It turns out that this allows for a new family of spectral estimates based on the theory presented in Section~\ref{sec:gen-set}.

\begin{defi}
Let $\Graph$ be a metric graph. A \emph{cycle double cover} of $\Graph$ is a double cover whose elements are cycles within $\Graph$.
\end{defi}

In the following, we are going to discuss estimates on the spectral gap of \emph{planar} quantum graphs. Given an embedding of the graph $\mG$ in the plane \(\R^2\) (or in the unit sphere \(\mathbb S^2\)), so that any two edges do not cross, $\mG$ and hence \(\mathcal{G}\) decompose the plane {(or sphere)} into a finite number of open sets \((\mf_i)_{1\leq i\leq k}\) -- the \emph{faces} of $\Graph$.
Hence, in the case of bridgeless planar graphs, each face $\mf_i$ defines an element \(\mathcal{C}_i\) of a cycle double cover \(\mathfrak C\): let \(\mathcal C_i\) be the cycle enclosing \(\mf_i\). The vicinity graph is a weighted version of what is known as the \textit{dual graph} in (combinatorial) graph theory: we hence denote it by $\mG_d$. The weight of the edge between \(i\neq j\) is given by the total length of the boundary shared by $\mf_i,\mf_j$. 

Because the elements $(\mathcal C_i)_{1\le i\le k}$ of the double cover are cycles, they  satisfy
\begin{equation}\label{eq:lambda2cyc}
\lambda_2(\mathcal C_i)= \frac{4\pi^2}{|\mathcal C_i|^2},\qquad i=1,\ldots,k.
\end{equation}
Applying Theorem~\ref{thm:amicoh-basic} to the elements of the cycle double cover immediately yields the following.

\begin{prop}\label{prop:amicoh-cdc}
Given a planar metric graph $\Graph$ with a cycle double cover $\mathfrak C:=(\mathcal C_i)_{1\le i\le k}$ and associated dual graph $\mG_d$, the $k$ lowest eigenvalues of the Laplacian with natural conditions on $\Graph$ satisfy
\begin{equation}\label{eq:amicoh-d}
\lambda_i(\Graph)\ge 2\pi^2 \alpha_i(\mG_d)\min_{1\le j\le k}\frac{1}{|\mathcal C_j|^2},\qquad i=1,\ldots,k.
\end{equation}
\end{prop}

As usual, $\alpha_i(\mG_d)$ denotes the $i$-th-lowest eigenvalue of the normalized Laplacian of $\mG_d$.

 It is not too restrictive to assume a metric graph to be bridgeless: indeed, doubling each edge of $\Graph$ yields a new, bridgeless graph $\Graph_2$ whose spectral gap is no larger than the spectral gap of $\Graph$.


Because each metric graph with a double cover also admits the double cover based on stars introduced in~\cite{AmiCoh18}, it is a natural question whether either of these two estimates is consistently better. One may think that the trade-off is apparent: cycle double covers tend to have much higher $\eta$, even though the eigenvalues $\alpha_i$ can be a bit smaller. On the other hand, the topology of $\Gamma$ is more obviously related to that of $\Graph$ if the star double cover is chosen. 

It turns out that the estimate Proposition~\ref{prop:aminicoh} can actually be quite efficient in many different examples. In order to discuss this issue,  we restrict for a while to the case of regular polyhedra.

\begin{exa}
{
Given a polyhedron $\mP$ on $V$ vertices with Schläfli symbol $\{n, m\}$ (i.e., each face is a regular $n$-gon, each vertex has degree $m$) and dual polyhedron $\mP_d$, we consider the \emph{equilateral metric polyhedron} $\mathcal P$ by assigning unitary length to each edge in $\mP$. Then the lower estimate in~\eqref{eq:amicoh-2} becomes
\begin{equation}\label{eq:sdc-schl}
\lambda_i(\mathcal P)\ge \frac{\pi^2}{8}\alpha_i(\mP),\qquad i=1,\ldots,V.
\end{equation}
If instead a cycle double cover based on the $F=\frac{4n}{2m+2n-mn}$ faces of the polyhedron is taken, then we promptly obtain the estimate
\begin{equation}\label{eq:cdc-schl}
\lambda_i(\mathcal P)\ge \frac{2\pi^2}{n^2}\alpha_i(\mP_d),\qquad i=1,\ldots,F.
\end{equation}

Hence,~\eqref{eq:cdc-schl} is sharper than~\eqref{eq:sdc-schl} if and only if 
\begin{equation}\label{eq:16:9}
\alpha_i(\mP)\le \frac{16}{n^2}\alpha_i(\mP_d).
\end{equation}


Let us test these estimates in the case of equilateral metric graphs built upon Platonic solids: they turn out to be a good benchmark for spectral bounds, since the actual eigenvalues can be found by means of von Below's formula \cite{Bel85}.
}

1) In order to apply our construction we consider a metric graph \(\mathcal{P}\) whose underlying combinatorial graph \(\mP\) has the shape of an icosahedron (cf.\ Figure \ref{fig:icosa}), i.e., the Platonic solid with Schläfli symbol \(\{3,5\}\).
	
	\begin{figure}[H]
\begin{minipage}[t]{0.5\textwidth}
\centering
\begin{tikzpicture}[scale=0.65]
\coordinate (v1) at (90:3.7);
\coordinate (v2) at (150:4);
\coordinate (v3) at (210:4);
\coordinate (v4) at (270:3.7);
\coordinate (v5) at (330:4);
\coordinate (v6) at (30:4);

\coordinate (v7) at ($(0,-0.3)+(150:2)$);
\coordinate (v8) at ($(0,-0.3)+(270:2)$);
\coordinate (v9) at ($(0,-0.3)+(30:2)$);

\coordinate (v10) at ($(0,0.3)+(90:2)$);
\coordinate (v11) at ($(0,0.3)+(210:2)$);
\coordinate (v12) at ($(0,0.3)+(330:2)$);

\draw[thick] (v1) -- (v2) -- (v3) -- (v4) -- (v5) -- (v6)-- (v1);
\draw[thick] (v7) -- (v8) -- (v9) -- (v7);
\draw[dotted] (v10) -- (v11) -- (v12) -- (v10);

\draw[thick] (v2) -- (v7) -- (v1)--(v9);
\draw[thick] (v7) -- (v3) -- (v8) -- (v4);
\draw[thick] (v8) -- (v5) -- (v9) -- (v6);

\draw[dotted] (v5) -- (v12) -- (v6) -- (v10) -- (v1);
\draw[dotted] (v10) -- (v2) -- (v11) -- (v3);
\draw[dotted] (v11) -- (v4) -- (v12);

\end{tikzpicture}
\vspace{5pt}
\caption{The icosahedron~...}\label{fig:icosa}
\end{minipage}%
\begin{minipage}[t]{0.5\textwidth}
\centering
\begin{tikzpicture}[scale=1.05]
\coordinate (f1) at (0,0);

\coordinate (f2) at (150:1.5);
\coordinate (f3) at (270:1.5);
\coordinate (f4) at (30:1.5);

\coordinate (f5) at (90:1.5);
\coordinate (f6) at (210:1.5);
\coordinate (f7) at (330:1.5);

\coordinate (a1) at (10:2.4);
\coordinate (a2) at (50:2.4);
\coordinate (a3) at (130:2.4);
\coordinate (a4) at (170:2.4);
\coordinate (a5) at (250:2.4);
\coordinate (a6) at (290:2.4);

\coordinate (b1) at (70:2.4);
\coordinate (b2) at (110:2.4);
\coordinate (b3) at (190:2.4);
\coordinate (b4) at (230:2.4);
\coordinate (b5) at (310:2.4);
\coordinate (b6) at (350:2.4);

\draw[thick] (f1)--(f5);
\draw[thick] (f1)--(f6);
\draw[thick] (f1)--(f7);

\draw[thick] (f7)--(a1)--(a2)--(f5)--(a3)--(a4)--(f6)--(a5)--(a6)--(f7);

\draw[thick] (a5)--(b4)--(b3)--(a4);
\draw[thick] (a3)--(b2)--(b1)--(a2);
\draw[thick] (a1)--(b6)--(b5)--(a6);

\draw[dotted] (f1)--(f2);
\draw[dotted] (f1)--(f3);
\draw[dotted] (f1)--(f4);

\draw[dotted] (b1)--(f4)--(b6);
\draw[dotted] (b5)--(f3)--(b4);
\draw[dotted] (b3)--(f2)--(b2);

\end{tikzpicture}
\vspace{5pt}
\caption{...~and its dual graph -- the dodecahedron}\label{fig:dodeca}
\end{minipage}
\end{figure}

 For simplicity we assume all edges in \(\mathcal{P}\) to have unit length. Then, we can use von Below's formula to calculate the eigenvalues of the Laplacian on \(\mathcal{P}\): first of all, we can compute the eigenvalues of the normalized Laplacian on the icosahedron \(\mP\)  explicitly. For instance, its spectral gap is
		\[\alpha_2(\mP)=\frac{5-\sqrt{5}}{5}\]
	and, therefore, von Below's formula~\cite{Bel85} implies
		\[\lambda_2(\mathcal{P})=\arccos(1-\alpha_2(\mP))^2=\arccos\left(\frac{\sqrt{5}}{5}\right)^2\simeq 1.226.\]
	Again, note that applying von Below's formula would not be possible if \(\mathcal{P}\) were not equilateral -- however, the following calculation can be extended to the non-equilateral case.\\
	We compare the estimates for star double covers and cycle double cover: the dual graph $\mP_d$ of \(\mP\) has the shape of the Platonic solid with Schläfli symbol \(\{5,3\}\) -- the dodecahedron (cf.~Figure \ref{fig:dodeca}). The spectral gap of the normalized Laplacian on the dodecahedron \(\mP_d\) is
		\[\alpha_2(\mP_d)=\frac{3-\sqrt{5}}{3}.\]
	Applying the estimate~\eqref{eq:cdc-schl} for cycle double covers yields the lower bound
		\[\lambda_2(\mathcal P)\geq \frac{2\pi^2}{3^2}\, \alpha_2(\mP_d)=\frac{2\pi^2(3-\sqrt{5})}{27}\simeq 0.558.\]
	In comparison, using the star-double-cover-based estimate in \eqref{eq:sdc-schl} we find
		\[\lambda_2(\mathcal{P})\geq \frac{\pi^2}{8}\alpha_2(\mP)= \frac{\pi^2(5-\sqrt{5})}{40}\simeq 0.682,\]
	whereas the estimates in~~\cite[Thm.~1.1]{Ken20} and \cite[Thm.~2.1]{BanLev17} yield the weaker lower bounds
	\begin{align*}
		\lambda_2(\mathcal{P}) & \geq \frac{1}{90}\simeq 0.011 & \lambda_2(\mathcal{P}) & \geq \frac{\pi^2}{225}\simeq 0.044,
	\end{align*}
	respectively, since the total length and diameter of \(\mathcal{P}\) are \(L=30\) and \(\mathcal D=3\).
	
2)	Let us take a look at higher eigenvalues: the eigenvalues of the icosahedron are
	\[
	0^1,\left( \frac{5-\sqrt{5}}{5}\right)^3, \left( \frac{6}{5}\right)^5,\left( \frac{5+\sqrt{5}}{5}\right)^3.
	\]
	
 On the other hand, the eigenvalues of the dodecahedron are
\[
0^1, \left( \frac{3-\sqrt{5}}{3}\right)^3,\left(\frac{2}{3} \right)^5, 1^4,\left(\frac{5}{3} \right)^4,\left( \frac{3+\sqrt{5}}{3}\right)^3,
\]	
so by~\eqref{eq:16:9} we see that the estimate~\eqref{eq:amicoh-2} based on the star double cover is better for the lowest eight nontrivial eigenvalues; whereas $\frac{5+\sqrt{5}}{5}\simeq 1.447$, thus the remaining eleven estimates yielded by the cycle double cover are better. The last relevant estimate reads
\[
\lambda_{18}(\mathcal P)\ge \frac{ (3+\sqrt{5})\pi^2}{27}\simeq 1.914.
\]
 For comparison: Friedlander's inequality~\eqref{eq:fried} yields
\[
\lambda_{18}(\mathcal P)\ge \frac{ 324\pi^2}{3600}\simeq 0.888.
\]

Conversely, if we take $\mathcal P$ to be the metric graph with the shape of the dodecahedron -- hence $\mP$ is the dodecahedron and $\mP_d$ is the icosahedron --, we see that the estimate based on the cycle double cover is strictly better if and only if
\[
\alpha_i(\mP)\le \frac{16}{25}\alpha_i(\mP_d):
\]
this is seen to be the case for the eight lowest non-trivial eigenvalues, whereas 
\[
1>\frac{16}{25}\left( \frac{5+\sqrt{5}}{5}\right)\simeq 0.926
\]
and hence the estimate
\[
\lambda_9(\mathcal P)\ge \frac{\pi^2}{8}
\]
obtained by means of Proposition~\ref{prop:aminicoh} is better.

3) Similar tests can be performed for the remaining regular polyhedra, too: in the case of the tetrahedron, the self-dual regular Platonic solid with Schläfli symbol \(\{3,3\}\) and eigenvalues
\[
0^1, \left(\frac{4}{3}\right)^3,
\]
the lower estimates on the spectral gap of the quantum graph based on Proposition~\ref{prop:amicoh-cdc} is trivially better than the estimate based on Proposition~\ref{prop:aminicoh}: they yield
\[
\lambda_2(\mathcal P)\ge \frac{8 \pi^2}{27}\quad\hbox{and}\quad \lambda_2(\mathcal P)\ge \frac{\pi^2}{6},
\]
respectively.

Let us finally turn to the Platonic solid with Schläfli symbol \(\{4,3\}\) -- the cube -- and its dual with Schläfli symbol \(\{3,4\}\) -- the octahedron. They have eigenvalues
\[
0^1,\left(\frac{2}{3}\right)^3,\left(\frac{4}{3}\right)^3,2^1\qquad\hbox{and}\qquad 0^1,1^3,\left(\frac{3}{2}\right)^2,
\]
respectively: hence, the estimate on $\lambda_2$ (resp., $\lambda_5$) of the cube based on Proposition~\ref{prop:amicoh-cdc} is better (resp., worse) then that based on Proposition~\ref{prop:aminicoh}. Instead, for the octahedron all estimates based on Proposition~\ref{prop:amicoh-cdc} are better than those based on  Proposition~\ref{prop:aminicoh}.

4) Different cycle double covers are conceivable: one may e.g.\ think of double cover elements defined by \textit{two} adjacent faces. In the case of the tetrahedron this would lead to three double cover elements consisting of diamonds bounded by edges $\me_1,\me_3,\me_4,\me_6$; $\me_2,\me_3,\me_4,\me_6$; and $\me_1,\me_2,\me_4,\me_5$. The corresponding vicinity graph is a cycle of length 3 whose eigenvalues are 
\[
0^1,\left(\frac{3}{2}\right)^2,
\]
whence the lower estimate 
\[
\lambda_2(\mathcal P)\ge \frac{3\pi^2}{16}.
\]
This bound is not as good as the one obtained from ``classical'' cycle double covers in 3).

5) 	As an example for an \(m\)-fold cover of order \(m>2\) let us, again, consider the equilateral graph \(\mathcal P\) corresponding to the cube. We consider the cover \(\mathfrak C\) of cycles that bound the union of two adjacent faces in the cube respectively. There are \(12\) such cycles, each having length \(6\), and \(\mathfrak C\) is a cycle \(6\)-fold cover of \(\mathcal P\). The corresponding vicinity graph \(\Gamma\) is a complete graph where each edge vertex has \(3\) incident edges with weight \(2\) and \(8\) incident edges with weight \(3\). Since the vicinity graph is complete and close to being equilateral we obtain a rather high spectral gap \(\alpha_2(\Gamma)\): namely the eigenvalues are
\[
0^1,\left(\frac{16}{15}\right)^9,\left(\frac{6}{5}\right)^2.
\]
Applying Theorem \ref{thm:amicoh-basic} directly to this cycle \(6\)-fold cover, we obtain the estimates
	\[\lambda_2(\mathcal P)\geq \frac{8\pi^2}{81},\qquad \lambda_{11}(\mathcal P)\geq \frac{\pi^2}{9}.\]
The estimate for \(\lambda_2(\mathcal P)\) is weaker than the one obtained from the cycle double cover corresponding to the dual platonic solid. However, note that \(\alpha_2(\Gamma)\) is higher than \(\alpha_2(\mP_d)\), whereas \(\eta\) is higher for the cycle double cover, since we only consider cycles of length \(4\) in that case.
\end{exa}

\begin{rem}
1) When applying~\eqref{eq:amicoh-d} to a concrete graph $\Graph$ it would be useful to have further information on the eigenvalues of $\mG_d$. Unfortunately, we are not aware of any abstract description of the normalized Laplacian on the dual of planar graphs and, in fact, it is unlikely that such a description is, generally, available at all: for already in the simple case of planar graphs, there is not a unique dual graph. Instead, a planar graph \(\mG\) generally depends on the choice of the embedding of \(\mG\) in the plane. In particular, applying Proposition~\ref{prop:amicoh-cdc} with the cycle double cover \((\mathcal C_i)_{1\leq i\leq k}\) may lead to different estimates depending on the choice of the embedding (cf.~ Figure \ref{fig:graph-diff-emb}). However, note that this situation may not occure when \(\mG\) is planar and \(3\)-connected because in that case \(\mG\) has a unique dual graph by a theorem of Whitney \cite[Theorem 11]{Whi32}.
 
\begin{figure}[H]
\begin{minipage}[t]{0.5\textwidth}
\centering
\begin{tikzpicture}
\coordinate (v1) at (45:2);
\coordinate (v2) at (135:2);
\coordinate (v3) at (225:2);
\coordinate (v4) at (315:2);
\coordinate (v5) at (0:{2*sqrt(2)});
\coordinate (v6) at (180:{2*sqrt(2)});

\coordinate (w1) at (0,0);
\coordinate (w2) at (0:{3/2*sqrt(2)});
\coordinate (w3) at (180:{3/2*sqrt(2)});
\coordinate (w4) at (270:{3/2*sqrt(2)});

\draw[dashed] (w2) -- node[midway,above]{\(1\)} (w1) -- node[midway,above]{\(1\)} (w3);
\draw[dashed] (w1) -- (w4) node[midway,left]{\(2\)};
\draw[dashed, in=0,out=-40] (w2) edge node[midway,below]{\(2\)} (w4); 
\draw[dashed,in=180,out=220] (w3) edge node[midway,below]{\(2\)} (w4);

\draw[thick] (v1) -- (v2) -- (v3) -- (v4) -- (v1);
\draw[thick] (v1) -- (v5) -- (v4);
\draw[thick] (v2) -- (v6) -- (v3); 

\node at (w1) [circle,draw,fill=white,scale=0.6] {};
\node at (w2) [circle,draw,fill=white,scale=0.6] {};
\node at (w3) [circle,draw,fill=white,scale=0.6] {};
\node at (w4) [circle,draw,fill=white,scale=0.6] {};

\node at (v1) [circle,draw,fill=black,scale=0.3] {};
\node at (v2) [circle,draw,fill=black,scale=0.3] {};
\node at (v3) [circle,draw,fill=black,scale=0.3] {};
\node at (v4) [circle,draw,fill=black,scale=0.3] {};
\node at (v5) [circle,draw,fill=black,scale=0.3] {};
\node at (v6) [circle,draw,fill=black,scale=0.3] {};

\end{tikzpicture}
\end{minipage}%
\begin{minipage}[t]{0.5\textwidth}
\centering
\begin{tikzpicture}
\coordinate (v1) at (45:2);
\coordinate (v2) at (135:2);
\coordinate (v3) at (225:2);
\coordinate (v4) at (315:2);
\coordinate (v5) at (0:{2*sqrt(2)});
\coordinate (v6) at (0,0);

\coordinate (w1) at (45:{1/2*sqrt(2)});
\coordinate (w2) at (0:{3/2*sqrt(2)});
\coordinate (w3) at (180:{1/2*sqrt(2)});
\coordinate (w4) at (0,-{3/2*sqrt(2)});

\coordinate (h1) at (-1.7,-{sqrt(2)}); 

\draw[dashed,out = 180,in = 30] (w1) edge node[midway,above] {\(2\)} (w3);
\draw[dashed,in = 0, out = 140] (w2) edge node[pos=0.7,above] {\(1\)} (w1);
\draw[dashed,out=-110,in=90] (w1) edge node[midway,right] {\(2\)} (w4);
\draw[dashed, in=0,out=-40] (w2) edge node[midway,below] {\(2\)} (w4);
\draw[dashed,in=90,out=210] (w3) edge node[pos=0.9,left] {\(1\)} (h1);
\draw[dashed,in=180,out=-90] (h1) edge (w4);

\draw[thick] (v1) -- (v2) -- (v3) -- (v4) -- (v1);
\draw[thick] (v1) -- (v5) -- (v4);
\draw[thick] (v2) -- (v6) -- (v3);

\node at (v1) [circle,draw,fill=black,scale=0.3] {};
\node at (v2) [circle,draw,fill=black,scale=0.3] {};
\node at (v3) [circle,draw,fill=black,scale=0.3] {};
\node at (v4) [circle,draw,fill=black,scale=0.3] {};
\node at (v5) [circle,draw,fill=black,scale=0.3] {};
\node at (v6) [circle,draw,fill=black,scale=0.3] {};


\node at (w1) [circle,draw,fill=white,scale=0.6] {};
\node at (w2) [circle,draw,fill=white,scale=0.6] {};
\node at (w3) [circle,draw,fill=white,scale=0.6] {};
\node at (w4) [circle,draw,fill=white,scale=0.6] {};
\end{tikzpicture}
\end{minipage}
\vspace{5pt}
\caption{Two embeddings of the same equilateral graph whose induced cycle double covers define different vicinity graphs.}\label{fig:graph-diff-emb}
\end{figure}
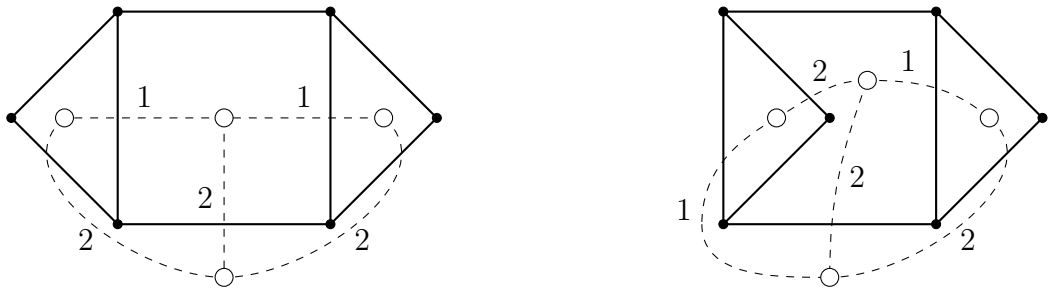

{
2) As we have seen, the estimate~\eqref{eq:amicoh} for $m=2$, which is based on double covers consisting of (pumpkin) stars,  is in the case of bridgeless graphs possibly improved by~\eqref{eq:amicoh-d}, which deals with double covers consisting of chains of cycles. The crucial point for this possible improvement is however that -- instead of an estimate on the spectral gap of stars -- the isoperimetric inequality for doubly connected graphs~\cite[Thm.~2.1]{BanLev17} can be used in order to estimate $\eta$, in case we are dealing with a \emph{cycle} double cover. Thus, the assertion in Theorem~\ref{eq:amicoh-d} carries over to the case of a double cover each of whose elements is merely a doubly connected subgraph of $\Graph$. We remind that a sharper -- if more technical -- estimate on $\eta$ is available in this case, see~\cite[Prop.~3.4]{BerKenKur17}.

3) More generally, Theorem~\ref{thm:amicoh-basic} can be applied with respect to double covers consisting of faces whenever a well-defined notion of duality exists.  This is e.g.\ the case for higher dimensional polytopes with respect to the canonical duality~\cite{Wen03} and for embedded graphs with respect to Petrie duality~\cite{Wil79}.

4) It is also possible to define $m$-fold covers of homogeneous simplicial $k$-complexes based on facets (i.e., on maximal simplices: $(k-1)$-dimensional hypertetrahedra); or else on all simplices of dimension $\le k-1$. Indeed, relations between normalized Laplacian eigenvalues of different dimensions are known, see~\cite[Thms.~5.1 and 5.3]{HorJos13b}.
 We do not go into details.
}
\end{rem}

We have outlined that estimates based on cycle double covers seem to perform very erratically in comparison with different double covers. Unfortunately, we do not have a cogent explanation for this phenomenon. A careful analysis of the proof of Theorem~\cite[Thm.~1.2]{AmiCoh18} suggests that the closer  the $i$-th eigenvalue of the Laplacian is to its mean value on each $\mathcal {\mathcal U}_i$, the more accurate the estimates on $\lambda_i(\Graph)$ based on the double cover $(\mathcal {\mathcal U}_i)_{i\in I}$ are. Hence, even a partial qualitative knowledge of the eigenfunction may help design a more convenient double cover that, in turn, allows for sharper estimates. An important class of metric graphs whose eigenfunctions are relatively well understood will be discussed in the next section.

\section{Generalized cycle double covers and pumpkin chains}\label{sec:pumpk}

After discussing in detail regular polyhedra, in this section we are going to consider a further class of graphs that does trivially satisfy the cycle double cover conjecture: \emph{pumpkins} and, more generally, \emph{pumpkin chains}.


%

An \(m\)-pumpkin \(\mathcal P\) is by definition a metric graph consisting of two vertices connected by arbitrarily (but finitely) many parallel edges $\me_i$ with edge lengths $\ell_i\in (0,\infty)$.

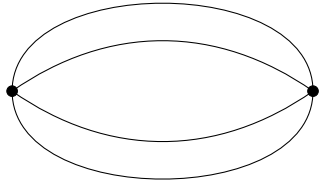
\begin{figure}[H]
\begin{tikzpicture}[scale=1]
\coordinate (a) at (0,0);
\coordinate (b) at (4,0);
\draw[fill] (a) circle (2pt);
\draw[fill] (b) circle (2pt);
\draw[bend left=35] (a) edge (b);
\draw[bend right=35] (a) edge (b);
\draw[bend left=90] (a) edge (b);
\draw[bend right=90] (a) edge (b);
\end{tikzpicture}
\vspace{5pt}
\caption{ A pumpkin}
\end{figure}

The eigenvalues of the Laplacian on an equilateral pumpkin can be derived explicitly. For non-equilateral pumpkins the eigenvalue problem becomes more difficult; furthermore, Theorem~\ref{thm:amicoh-basic} only yields a tautology whenever applied to a star double cover, as observed in Remark~\ref{rem:tautol}.
 However, we shall demonstrate how to apply Theorem \ref{thm:amicoh-basic} to cycle double covers of \(\mathcal P\) to obtain a lower bound for eigenvalues of the Laplacian on \(\mathcal{P}\). To define a cycle double cover \(\mathfrak{C}=(\mathcal C_i)_{1\leq i\leq m}\) for \(\mathcal P\) we consider the cycles given by \(\mathcal C_i=\me_i\cup \me_{i+1}\) for \(i=1,\ldots,m-1\) and $\mathcal C_m=\me_m\cup \me_1\). Note that pumpkins obviously are planar and the just constructed double cover is the one described in the previous section for a specific plane embedding of \(\mathcal P\). The vicinity graph \(\Gamma\) induced by \(\mathfrak{C}\) is itself a discrete cycle, where \(i\) and \(i+1\) are adjacent for \(i=1,\ldots,m-1\) and the edge connecting \(i\) and \(i+1\) has the edge weight \(\ell_i\) and the edge connecting \(m\) and \(1\) has weight \(\ell_m\).
Note that, depending on the ordering of the edge set \(\{\me_1,\ldots,\me_m\}\), we obtain different configurations of the cycle double cover. In fact, there are \(\frac{(m-1)!}{2}\) such configurations for \(\mathcal P\) if all edges have different lengths.

\begin{exa}
In order to compare different double cover configurations for a given non-equilateral pumpkin we consider a \(4\)-pumpkin \(\mathcal{P}\) consisting of two edges \(\me_1,\me_1'\) of length \(1\) and two edges \(\me_a,\me_a'\) of variable length \(a\geq 1\) respectively. Following the construction above we find exactly two cycle double covers \(\mathfrak{C}_1\) and \(\mathfrak{C}_2\) of \(\mathcal{P}\) given by
	\begin{align*}
		\mathfrak{C}_1 & =(\me_1\cup\me_1',\me_1'\cup\me_a',\me_a'\cup\me_a,\me_a\cup\me_1),
		& \mathfrak{C}_2 & =(\me_1\cup\me_a,\me_a\cup\me_1',\me_1'\cup\me_a',\me_a'\cup\me_1).
	\end{align*}
	It can be shown that the spectral gap of the normalized Laplacians corresponding to the respective vicinity graphs \(\Gamma_i\) are given by
	\begin{align*}
		\alpha_2(\Gamma_1) & =1,
		& \alpha_2(\Gamma_2) & =\frac{2}{a+1}.
	\end{align*}
	Furthermore, the minimal spectral gaps \(\eta_i\) of the single covering elements are
	\begin{align*}
		\eta_1 & =\frac{\pi^2}{a^2}, & \eta_2 & =\frac{4\pi^2}{(a+1)^2}
	\end{align*}
	and Theorem \ref{thm:amicoh-basic} yields the lower bounds 
	\begin{align}\label{eq:estimate-4pumpkin}
		\lambda_2(\mathcal P) & \geq \frac{\pi^2}{2a^2}, & \lambda_2(\mathcal P) & \geq \frac{4\pi^2}{(a+1)^3}.
	\end{align}
	One quickly checks that the first estimate gives a better lower bound for \(\lambda_2(\mathcal P)\) if and only if \(a> 2+\sqrt 5\). This is likely due to the fact that the largest cycle of our \(4\)-pumpkin has length \(2a\) and the eigenfunction corresponding to the actual spectral gap \(\lambda_2(\mathcal P)=\frac{\pi^2}{a^2}\) is only supported by this cycle. So, one might actually expect that the cycle \(\me_a\cup\me_a'\) is an element of the ``optimal'' double cycle cover for \(\lambda_2(\mathcal P)\).
	
	Even in this simple example the choice of an optimal double cover depends on the ratio of the single edge lengths and the choice seems to be more involved if the number of edges increases.
\end{exa}

A \emph{pumpkin chain} arises by taking an interval, subdividing it into $n$ pieces by inserting $n-1$ vertices $\mv_{1},\ldots,\ldots,\mv_{n-1}$, and adding arbitrarily (but finitely) many parallel edges between any two consecutive vertices; we refer to~\cite[\S~5]{KenKurMal16} for a precise definition, but the following picture will probably suffice.

\begin{figure}[H]
\begin{tikzpicture}[scale=1]
\coordinate (a) at (0,0);
\coordinate (b1) at (1,1);
\coordinate (b) at (2,0);
\coordinate (c) at (3.5,0);
\coordinate (d1) at (4.75,1);
\coordinate (d) at (6,0);
\draw[fill] (a) circle (2pt);
\draw[fill] (b) circle (2pt);
\draw[fill] (c) circle (2pt);
\draw[fill] (d) circle (2pt);
\draw (a) -- (b);
\draw[bend left=45] (a) edge (b);
\draw[bend right=45] (a) edge (b);
\draw[bend left=60] (b) edge (c);
\draw[bend right=60] (b) edge (c);
\draw[bend left=45] (c) edge (d);
\draw[bend right=45] (c) edge (d);
\draw[bend left=20] (c) edge (d);
\draw[bend right=20] (c) edge (d);
\end{tikzpicture}
\vspace{5pt}
\caption{ A pumpkin chain}
\label{fig:shifting-pumpkins}
\end{figure}
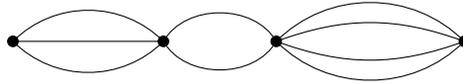

If all parallel edges within each pumpkin have the same length, we call a pumpkin chain \emph{locally equilateral}.

It has become increasingly clear over the last few years that pumpkin chains represent an important class of metric graphs. Locally equilateral pumpkin chains enjoy several good properties (in particular, the eigenfunction corresponding to the spectral gap is monotonic along the chain) and they are hence often used as reference graphs in geometric spectral theory for quantum graphs: we mention two relevant justifications in form of reductions obtained in~\cite[Lemma~5.4]{KenKurMal16} and~\cite[Lemma~5.3]{BerKenKur19}.


\begin{lemma}
\label{lem:locally-equilateral}
  Let $\Graph$ be compact and connected.
  Then there exist two locally equilateral pumpkin chains $\mathcal{P}_1,\mathcal P_2$, both with total length no larger than that of $\Graph$, such that
  \begin{displaymath}
\lambda_2 (\mathcal{P}_2)\geq	\lambda_2 (\Graph) \ge  \lambda_2 (\mathcal{P}_1).
  \end{displaymath}
  Furthermore, $\mathcal P_2$ can be taken to have same diameter as $\Graph$, whereas the bridges within $\mathcal P_1$ can be taken to have total length no longer than the bridges within $\Graph$.
\end{lemma}

While their eigenvalues and eigenfunctions are fairly well understood, see~\cite[\S~5]{BerKenKur19}, there are hardly any spectral estimates for pumpkin chains.

Let us apply our theory in order to deduce lower bounds on the spectral gap of pumpkin chains.
A first estimate can be directly derived from Proposition~\ref{prop:aminicoh}: it holds for pumpkin chains that are not necessarily locally equilateral.
\begin{prop}\label{prop:diam}
A chain \(\mathcal P\) of $n$ pumpkins \(\mathcal P_1,\ldots,\mathcal P_n\) admits the lower bound
	\begin{equation}\label{eq:acs-cdc-pmax}
	\lambda_2(\mathcal P)\ge \frac{\pi^2}{4\ell_{\max}^2}
		\frac{1}{\sum\limits_{i=1}^n |{\mathcal P}_i|}\frac{1}{\sum\limits_{i=1}^n \frac{1}{|{\mathcal P}_i|}},
	\end{equation}
	where $\ell_{\max}$ is the maximal length of the edges in \(\mathcal P\) and $|\mathcal P_i|$ is the Lebesgue measure of the pumpkin $\mathcal P_i$.
\end{prop}
Needless to say, the Lebesgue measure of the pumpkin $\mathcal P_i$ is the total length of its edges; we casually refer to $|\mathcal P_i|$ as the \textit{total length} of $\mathcal P_i$.

\begin{proof}
	Applying Proposition \ref{prop:aminicoh} yields
		\[\lambda_2(\mathcal P)\geq \frac{\pi^2}{8\ell_{\max}^2}\alpha_2(\Gamma),\]
	where \(\Gamma\) is the vicinity graph corresponding to the star double cover of \(\mathcal P\). In fact, \(\Gamma\) is a path on \(n+1\) vertices with edge weights \(\mu_i:=|{\mathcal P}_i|\) for \(i=1,\ldots,n\). Thus, the volume of \(\Gamma\) with respect to the degree vertex weight and the diameter of \(\Gamma\) with respect to the inverse edge weight \(\mu^{-1}\) are
	\begin{align*}	
		\Vol_\mu(\Gamma) & =2\sum_{i=1}^n |\mathcal P_i|  & \D_{\mu^{-1}}(\Gamma) & =\sum_{i=1}^n\frac{1}{|{\mathcal P}_i|}.
	\end{align*}	
	Using the lower bound \eqref{eq:bounds-normlapl} proves the claim.
	\end{proof}


\begin{cor}
A chain \(\mathcal P\) of $n$ pumpkins \(\mathcal P_1,\ldots,\mathcal P_n\) admits the lower bound
	\begin{equation}\label{eq:harm}
		\lambda_2(\mathcal P)\geq \frac{4|\mathcal P|_{\min}|\mathcal P|_{\max}}{\left(|\mathcal P|_{\min}+|\mathcal P|_{\max}\right)^2}\cdot \frac{\pi^2}{4n^2\ell_{\max}^2},
	\end{equation}
	where \(\ell_{\max}\) is the maximal length the edges in \(\mathcal P\) and $|\mathcal P|_{\min},|\mathcal P|_{\max}$ are the maximal and minimal total lengths among those of \(\mathcal P_1,\ldots,\mathcal P_n\).
\end{cor}
\begin{proof}
Let us take a closer look at~\eqref{eq:acs-cdc-pmax}.
Because, $\frac{1}{n^2}\sum\limits_{i=1}^n |{\mathcal P}_i|\sum\limits_{i=1}^n \frac{1}{|{\mathcal P}_i|}$ is the quotient of the arithmetic and the harmonic mean of $(|\mathcal P_i|)_{1\le i\le n}$, by a known estimate (cf.~ \cite[\S 2.11]{Mit70})
\[
\sum\limits_{i=1}^n |{\mathcal P}_i|\sum\limits_{i=1}^n \frac{1}{|{\mathcal P}_i|}\le  
\frac{4n^2|\mathcal P|_{\max}|\mathcal P|_{\min}}{\left( |\mathcal P|_{\max}+|\mathcal P|_{\min}\right)^2}.
\]
This concludes the proof.
\end{proof}
The above estimates complement a recent \emph{upper} bound for the spectral gap of a chain of $n\ge 2$ pumpkins with total length $L$: it has been proved in~\cite{BorCorJon19} that 
\[
\lambda_2(\mathcal P)\le \frac{(n+1)^2 \pi^2}{4L^2}.
\]
\begin{rem}
Combining Proposition~\ref{prop:diam} with Lemma~\ref{lem:locally-equilateral} we may deduce a lower bound on the spectral gap of an arbitrary quantum graph, provided some information on the critical sets of an eigenfunction associated with the spectral gap is known; we omit the (rather technical) details and refer to the proof of~\cite[Lemma~5.3]{BerKenKur19}.
\end{rem}

The relevant idea that has led us to Proposition~\ref{prop:amicoh-cdc} is that the elements of a cycle double cover $ \mathfrak C = (\mathcal C_i) _ {1\leq i\leq k} $
satisfy \eqref{eq:lambda2cyc}. To this purpose, the elements of the double cover need actually not be cycles at all.

\begin{defi}
Let $\Graph$ be a metric graph. A \emph{generalized cycle double cover} of $\Graph$ is a double cover whose elements are doubly connected subgraphs of $\Graph$.
\end{defi}
The assertion of Proposition~\ref{prop:amicoh-cdc} carries over verbatim to metric graphs admitting such generalized cycle double covers: using
\begin{equation}\label{eq:lambda2cyc-bis}
\lambda_2(\mathcal C_i)\ge \frac{4\pi^2}{|\mathcal C_i|^2},\qquad i=1,\ldots,k,
\end{equation}
by~\cite[Thm.~2.1]{BanLev17} we obtain the following result.
\begin{prop}\label{prop:amicoh-gcdc}
Given a metric graph $\Graph$ with a generalized cycle double cover $\mathfrak C:=(\mathcal C_i)_{1\le i\le k}$ and associated vicinity graph $\Gamma$, the $k$ lowest eigenvalues of the Laplacian with natural conditions on $\Graph$ satisfy
\begin{equation}\label{eq:amicoh-gd}
\lambda_i(\Graph)\ge 2\pi^2 \alpha_i(\Gamma)\min_{1\le j\le k}\frac{1}{|\mathcal C_j|^2},\qquad i=1,\ldots,k.
\end{equation}
\end{prop}

Next, we provide two algorithms to define a generalized cycle double cover of a pumpkin chain $\mathcal P$, based on an arbitrary but fixed planar embedding of $\mathcal P$ consisting of $n$ linked pumpkins. We explain our constructions based on the example of the pumpkin chain in Figure~\ref{fig:shifting-pumpkins}.
Indeed,  pumpkin chains are planar but their cycle double covers lead to disconnected vicinity graphs: i.e., this choice leads, for instance, to $\alpha_2=0$ and hence to trivial eigenvalue estimates. However, it is easy to come up with different generalized cycle double covers of pumpkin chains, each leading to a different estimate based on Proposition~\ref{prop:aminicoh} or Proposition~\ref{prop:amicoh-cdc}.

In the following we denote by $m_i$ the number of parallel edges between $\mv_{i-1}$ and $\mv_{i}$ and we make the assumption that \(\mathcal P\) is bridgeless, i.e.~\(m_i\ge 2\) holds for \(1\le i\le n\): accordingly, the edges belonging to the $i$-th pumpkin $\mathcal P_i$ are $\me_{i,1},\ldots,\me_{i,m_i}$. Unless explicitly stated we assume the pumpkin chain to be locally equilateral and denote by $\ell_i$ the common length of all such $m_i$ edges. Nevertheless, we point that our constructions may also be applied to the more involved case of non-equilateral pumpkin chains.

To construct cycle generalized cycle double covers for \(\mathcal P\) we start by considering the cycle double covers of the single pumpkins \(\mathcal P_i\) constructed previously: we set
\begin{align}
	\mathcal C_{i,j}:=\me_{i,j}\cup\me_{i,j+1}, &\qquad i=1,\ldots,n,~j=1,\ldots, m_i-1\label{eq:cdc-allpumpkins1}\\
\intertext{and}
	\mathcal C_{i,m_i}:=\me_{i,m_i}\cup\me_{i,1},&\qquad i=1,\ldots,n.\label{eq:cdc-allpumpkins2}
\end{align}
This already defines a cycle double cover of \(\mathcal P\), but the corresponding vicinity graph is just the disjoint union of \(n\) (combinatorial) cycle graphs and, thus, it does not see the structure of the whole pumpkin chain anymore. Therefore, our aim is to ``glue'' the cycles \(\mathcal C_{i,j}\) to obtain generalized cycle double covers of \(\mathcal P\) that represent the structure of \(\mathcal P\) more appropriately.

\subsection{Layered double covering}

To obtain a generalized double cover \(\mathfrak{U}\) of an arbitrary pumpkin chain   $\mathcal P$ let us glue the cycles \(\mathcal C_{i,j}\) defined in \eqref{eq:cdc-allpumpkins1} and \eqref{eq:cdc-allpumpkins2} as follows:
\begin{itemize}
\item the first covering element consists of the pumpkin chain comprising all cycles that run along the ``bottom'' of $\Graph$: 
\[\bigcup_{1\leq i\leq n}\mathcal C_{i,1};\]
its spectral gap is $\frac{\pi^2}{\mathcal D^2}$, where $\mathcal D$ denotes the diameter of the pumpkin chain, i.e., $\mathcal D:=\sum_{i=1}^n \ell_i$. 
\item we continue with the following layer: 
\[\bigcup_{\substack{1\leq i\leq n,\\ 2<{m_i}}}\mathcal C_{i,2}\]
where we stipulate that we pass to the connected components if this metric subgraph is not connected, which might be the case when any of the cycles in this union is missing;
\item this process is repeated $(\max\limits_i m_i)-1$ times, thus exhausting all layers of the pumpkin chain with the covering elements:
\[
 \bigcup_{\substack{1\leq i\leq n,\\ j<{m_i}}}\mathcal C_{i,j}
\]
for \(j=1,\ldots,m_i-1\);
\item finally, we ``lock up'' each pumpkin with the final covering element:
\[
 \bigcup_{1\leq i\leq n}\mathcal C_{i,m_i}.
\]

\end{itemize}

Now, consider the graph in Figure~\ref{fig:shifting-pumpkins}: the vicinity graph that arises by means of the above construction is depicted in Figure~\ref{fig:vicin-layered}.
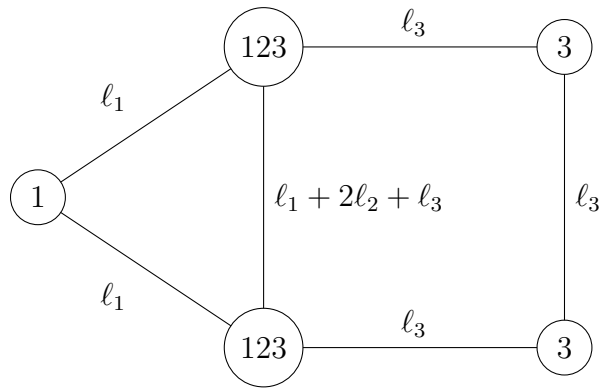
\begin{figure}[H]
\begin{tikzpicture}[scale=1]
\coordinate (a) at (0,0);
\coordinate (b) at (0,4);
\coordinate (c) at (4,4);
\coordinate (d) at (4,0);
\coordinate (e) at (-3,2);
\draw (a) -- (b)--(c)--(d) -- (a);
\draw (a) -- (e) -- (b);
\node at (0,2) [anchor=west] {$\ell_1+2\ell_2+\ell_3$};
\node at (4,2) [anchor=west] {$\ell_3$};
\node at (2,4) [anchor=south] {$\ell_3$};
\node at (2,0) [anchor=south] {$\ell_3$};
\node at (-2,1) [anchor=north] {$\ell_1$};
\node at (-2,3) [anchor=south] {$\ell_1$};
\node at (a) [circle,draw,fill=white] {123};
\node at (b) [circle,draw,fill=white] {123};
\node at (c) [circle,draw,fill=white] {3};
\node at (d) [circle,draw,fill=white] {3};
\node at (e) [circle,draw,fill=white] {1};
\end{tikzpicture}
\vspace{5pt}
\caption{ The layered graph of the graph in Figure~\ref{fig:shifting-pumpkins}. Each vertex is a cycle and the text in the vertex summarizes the pumpkins touched by the corresponding cycle.}
\label{fig:vicin-layered}
\end{figure}

Let us now look at the pumpkin chain obtained swapping the second and third pumpkin, thus obtaining the graph in Figure~\ref{fig:shifting-pumpkins-2}.

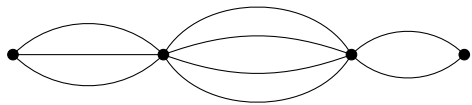
\begin{figure}[H]
\begin{tikzpicture}[scale=1]
\coordinate (a) at (0,0);
\coordinate (b) at (2,0);
\coordinate (c) at (4.5,0);
\coordinate (d) at (6,0);
\draw[fill] (a) circle (2pt);
\draw[fill] (b) circle (2pt);
\draw[fill] (c) circle (2pt);
\draw[fill] (d) circle (2pt);
\draw (a) -- (b);
\draw[bend left=45] (a) edge (b);
\draw[bend right=45] (a) edge (b);
\draw[bend left=60] (b) edge (c);
\draw[bend right=60] (b) edge (c);
\draw[bend left=45] (c) edge (d);
\draw[bend right=45] (c) edge (d);
\draw[bend left=20] (b) edge (c);
\draw[bend right=20] (b) edge (c);
\end{tikzpicture}
\vspace{5pt}
\caption{ A pumpkin chain obtained swapping pumpkins inside the graph in Figure~\ref{fig:shifting-pumpkins}. We refer to the original graph and regard the central pumpkin as pumpkin \#3.}
\label{fig:shifting-pumpkins-2}
\end{figure}

The ``layered'' vicinity graph becomes the one depicted in Figure~\ref{fig:vicin-layered-2}.

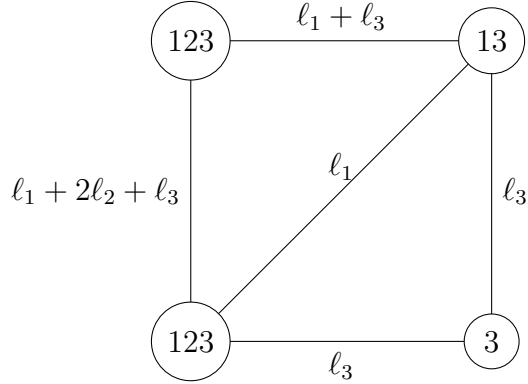
\begin{figure}[H]
\begin{tikzpicture}[scale=1]
\coordinate (a) at (0,0);
\coordinate (d) at (4,0);
\coordinate (c) at (4,4);
\coordinate (b) at (0,4);
\draw (a) -- (b)--(c)--(d) -- (a);
\draw (a) -- (c);
\node at (0,2) [anchor=east] {$\ell_1+2\ell_2+\ell_3$};
\node at (4,2) [anchor=west] {$\ell_3$};
\node at (2,4) [anchor=south] {$\ell_1+\ell_3$};
\node at (2,0) [anchor=north] {$\ell_3$};
\node at (2,2) [anchor=south] {$\ell_1$};
\node at (a) [circle,draw,fill=white] {123};
\node at (b) [circle,draw,fill=white] {123};
\node at (c) [circle,draw,fill=white] {13};
\node at (d) [circle,draw,fill=white] {3};
\end{tikzpicture}
\vspace{5pt}
\caption{ The layered vicinity graph of the graph in Figure~\ref{fig:shifting-pumpkins-2}  }
\label{fig:vicin-layered-2}
\end{figure}
If pumpkins are swapped in order to avoid disconnected layers, then the vicinity graph $\Gamma$ becomes smaller and tends to be more strongly connected, leading to improved estimates; but at the same time to estimates (via~\eqref{eq:amicoh}) only for a lower number of eigenvalues of $\mathcal P$.

\subsection{Concatenated double covering}

In layered double coverings the vicinity graph $\Gamma$ does not detect the position of the pumpkins: whether they are peripheral (which, based on considerations in~\cite[\S~5]{BerKenKur19}, should heuristically lead to lower $\alpha_2(\Gamma)$) or rather more central. This suggests yet one different approach to the task of double-covering a metric graph $\mathcal P$:
\begin{itemize}
\item the first \(n-1\) covering elements $\mathcal U_1,\ldots,\mathcal U_{n-1}$ consist of 
\[
 \mathcal U_i=\mathcal C_{i,1}\cup\mathcal C_{i+1,2},\quad i=1,\ldots,n-1,
\]
thus linking successive pumpkins within $\mathcal P$;
\item for the other elements of the double covering we choose the remaining cycles \(\mathcal C_{i,j}\) in \eqref{eq:cdc-allpumpkins1} and \eqref{eq:cdc-allpumpkins2} that do not appear in any of the elements \(\mathcal U_i\) constructed in the previous step.
\end{itemize}

The corresponding vicinity graph has $1-n+\sum_{i=1}^n m_i$ vertices; it consists of a ``backbone'' given by the vertices $\mathcal U_1,\ldots,\mathcal U_{n-1}$; by two (possibly degenerate) cycles on $m_i$ edges attached to $\mathcal U_1$ and $\mathcal U_{n-1}$ (corresponding to the cycles inside the first and the $n$-th pumpkin); and, for all $i=1,\ldots,n-1$, by paths on $m_i-1$ edges attached to both vertices ${\mathcal U}_i$ and $\mathcal U_{i+1}$, corresponding to the remaining cycles inside the $i$-th pumpkin.


\begin{figure}[H]
\begin{tikzpicture}[scale=1]
\coordinate (a) at (0,0);
\coordinate (b) at (-2,1);
\coordinate (c) at (-2,-1);
\coordinate (d) at (3,0);
\coordinate (e) at (5,2);
\coordinate (f) at (7,0);
\coordinate (g) at (5,-2);
\draw (a) -- (b)--(c)--(a);
\draw (a) -- (d);
\draw (d) -- (e)--(f)--(g)--(d);
\node at (1.5,0) [anchor=south] {$2\ell_2$};
\node at (-1,1) [anchor=west] {$\ell_1$};
\node at (-1,-1) [anchor=west] {$\ell_1$};
\node at (-2,0) [anchor=east] {$\ell_1$};
\node at (4,2) [anchor=north] {$\ell_3$};
\node at (6,2) [anchor=north] {$\ell_3$};
\node at (6,-2) [anchor=south] {$\ell_3$};
\node at (4,-2) [anchor=south] {$\ell_3$};
\node at (a) [circle,draw,fill=white] {12};
\node at (b) [circle,draw,fill=white] {1};
\node at (c) [circle,draw,fill=white] {1};
\node at (d) [circle,draw,fill=white] {23};
\node at (e) [circle,draw,fill=white] {3};
\node at (f) [circle,draw,fill=white] {3};
\node at (g) [circle,draw,fill=white] {3};
\end{tikzpicture}
\vspace{5pt}
\caption{ The concatenated vicinity graph of the graph in Figure~\ref{fig:shifting-pumpkins}}
\label{fig:vicin-concaten}
\end{figure}
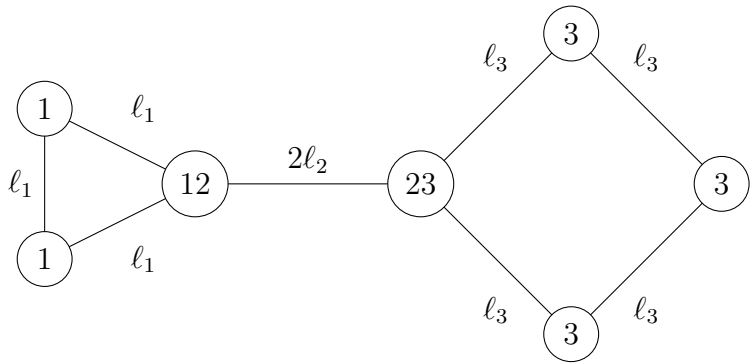

Intuitively, in comparison with the vicinity graph based on the layered construction, $\eta$ will be much larger but $\alpha_2$ will be smaller, too. In the example of the graph in Figure~\ref{fig:shifting-pumpkins}, if $\ell_1=\ell_2=\ell_3$, then 
\begin{itemize}
\item $\eta=\frac{\pi^2}{9}$ for the layered construction, while we have found the value $\alpha_2(\Gamma)\simeq 0.629$, leading to the estimate $\lambda_2(\mathcal P)\ge 0.345$;
\item $\eta=\frac{\pi^2}{4}$ for the concatenated construction, while we have found the value $\alpha_2(\Gamma)\simeq 0.229$, leading to the estimate $\lambda_2(\mathcal P)\ge 0.282$;
\end{itemize}
In the case of the pumpkin chain in Figure~\ref{fig:shifting-pumpkins} the estimates {in~\cite[Thm.~2.1]{BanLev17},~\cite[Thm.~1.1]{Ken20}, and Proposition~\ref{prop:diam}} yield
\begin{equation}\label{eq:blkp}
\lambda_2(\mathcal P)\ge 0.487, \quad\lambda_2(\mathcal P)\ge 0.055,\quad \hbox{and}\quad 
\lambda_2(\mathcal P)\ge 0.244,
\end{equation}
respectively.
If we pass to the graph in Figure~\ref{fig:shifting-pumpkins-2}, \(\alpha_2(\Gamma)\simeq 0.322\) for the concatenated construction, leading to the estimate \(\lambda_2\geq 0.398\), and $\alpha_2(\Gamma)\simeq 0.974$ for the layered construction,
leading to $\lambda_2\ge 0.533$; the latter is sharper than the estimates in~\eqref{eq:blkp}, which are all invariant under pumpkin swapping and hence do not distinguish between the graphs in Figure~\ref{fig:shifting-pumpkins} and~\ref{fig:shifting-pumpkins-2}.

\end{document}